\documentclass[12pt,twoside]{amsart}
\usepackage{amsmath}
\usepackage{amssymb}
\usepackage{amsthm}
\usepackage{eucal}

\newtheorem{theorem}{Theorem}[section]
\newtheorem{corollary}{Corollary}[section]
\newtheorem{proposition}{Proposition}[section]
\newtheorem{lemma}{Lemma}[section]
\numberwithin{equation}{section}

\theoremstyle{remark}

\newcommand{\cl}{\mathcal}
\newcommand{\bb}{\mathbb}

\newcommand{\ol}{\overline}
\newcommand{\wt}{\widetilde}
\newcommand{\ve}{\varepsilon}
\newcommand{\bs}{\boldsymbol}

\begin{document}

\title{Asymptotic Convergence of Solutions for One-Dimensional Keller-Segel Equations}

\author{Satoru Iwasaki, Koichi Osaki and Atsushi Yagi} 
\dedicatory{{\rm (Osaka University and Kwansei Gakuin University, Japan)}}

\maketitle

\begin{abstract}
The second and third authors of this paper have constructed in \cite{OY} finite-dimensional attractors for the one-dimensional Keller-Segel equations. They have also remarked in \cite[Section 7]{OY} that, when the sensitivity function is a linear function, the equations admit a global Lyapunov function. But at that moment they could not show the asymptotic convergence of solutions. This paper is then devoted to supplementing the results of \cite[Section 7]{OY} by showing that, as $t \to \infty$, every solution necessarily converges to a stationary solution by using the {\L}ojasiewicz-Simon gradient inequality of the Lyapunov function.
\par\smallskip

{\it Key Words and Phrases}.  Asymptotic convergence, Keller-Segel equations, Chemotaxis model.
\par\smallskip

2010 {\it Mathematics Subject Classification Numbers}\quad 35K45, 35B40, 92C17.
\end{abstract}

%%--------------------------
\section{Introduction}

We are concerned with the one-dimensional Keller-Segel equations
\begin{equation}  \label{1}
\left\{ \begin{aligned}
  &u_t = au_{xx} - k[u\rho_x]_x
       \hspace{1.5cm}\text{in}\quad I\times(0,\infty),   \\
  &\rho_t = b\rho_{xx} - d\rho + cu
       \hspace{1.5cm}\text{in}\quad I\times(0,\infty),   \\
  &u_x = \rho_x = 0
     \hspace{2.8cm}\text{on}\quad \partial I\times(0,\infty),
\end{aligned}  \right.
\end{equation}
together with the initial conditions
\begin{equation}  \label{2}
  u(x,0)=u_0(x) \quad\text{and}\quad \rho(x,0)=\rho_0(x)
  \qquad\text{in} \quad I,
\end{equation}
where $I = (\alpha,\beta)$ is a bounded open interval and $\partial I = \{\alpha,\beta\}$ is the set of its boundary points. Here, the unknown functions $u=u(x,t)$ and $\rho=\rho(x,t)$ denote the density of bacteria and the concentration of chemical substance, respectively, at position $x \in \ol I = [\alpha,\beta]$ and at time $t \in [0,\infty)$. The constants $a>0$ and $b>0$ are the diffusion coefficients of bacteria and chemical substance, respectively. The substance declines at a constant rate $d>0$ and is secreted at a constant rate $c>0$. The constant $k>0$ denotes the intensity of chemotaxis. (For a survey of the derivation of \eqref{1}, see, e.g.~\cite[Introduction]{Ya0}.)
\par

%%%  Motivation-Objective %%%
The second and third authors of this paper have constructed in \cite{OY} finite-dimensional attractors for the one-dimensional Keller-Segel equations of general form (including \eqref{1}-\eqref{2}). They have also remarked in \cite[Section 7]{OY} that, when the sensitivity function is a linear function as in \eqref{1}, the equations admit a global Lyapunov function. But at that moment they could not show the asymptotic convergence of solutions. The objective of this paper is then to supplement the results of \cite[Section 7]{OY} by showing that, as $t \to \infty$, the solution of \eqref{1}-\eqref{2} necessarily converges to a stationary solution of \eqref{1}.
\par

%%%  Gradient Inequality  %%%
As well known (see, e.g.~\cite[Section 1]{PM}), even if a system of ordinary differential equations admits a Lyapunov function, the system can possess some bounded solution whose $\omega$-limit set is a continuum. Similarly, there exists a nonlinear diffusion system admitting a Lyapunov function but possesses some globally bounded solution whose $\omega$-limit set is a continuum. For the systems of ordinary differential equations, {\L}ojasiewicz has presented in \cite{Lo} a sufficient condition for the Lyapunov function which is now called the {\L}ojasiewicz gradient inequality in order that every bounded solution converges to a stationary solution. The gradient inequality has been extended into an infinite-dimensional version by Simon \cite{Si} which is now called the {\L}ojasiewicz-Simon gradient inequality. By checking that the Lyapunov function admitted by a diffusion system satisfies the gradient inequality, it is possible to claim that its globally bounded solution necessarily converges to a stationary solution. After the paper \cite{Si} being published, many authors tried to devise more convenient ways how to check the {\L}ojasiewicz-Simon gradient inequality in the framework of Functional Analysis, see Chill \cite{Ch}, Chill-Haraux-Jendoubi \cite{CHJ}, Haraux-Jendoubi \cite{HJ}, Jendoubi \cite{Je}, and so on. In the meantime, Feireisl-Issard{\bf -}Roch-Petzeltov\'{a} \cite{FIP} has devised a non-smooth version of the gradient inequality.
\par

%%%  Nonsmooth Gradient Inequality  %%%
Essentially, we can use the methods developed by \cite{Ch,CHJ,HJ,Je} for our equations \eqref{1}-\eqref{2}, too. However, for a direct application, it must hold true that the Lyapunov function $\Phi(u,\rho)$ given by \eqref{16} for $0 \le u \in L_2(I)$ and $\rho \in H^1(I)$ is twice continuously Fr\'echet differentiable, which we cannot verify unfortunately. By this reason, we will introduce a modified version in which we require that $\Phi(u,\rho)$ is only once continuously Fr\'echet differentiable and its derivative $\Phi'(u,\rho)$ is only G\^ateaux differentiable with a derivative of bounded linear operator and will recover these weaker differentiabilities of $\Phi(u,\rho)$ by using other techniques described in Section 6.
\par

%%%%  Comparesion of both results  %%%%
An alternative way may be to use the methods for non-smooth Lyapunov functions devised by \cite{FIP}. In fact, using the methods, Feireisl-Lauren\c{c}ot-Petzeltov\'{a} have already proved in \cite{FLP} the asymptotic convergence for the classical solutions of \eqref{1}. But, as we constructed in \cite{OY} only strict solutions for \eqref{1}, we want to choose the other way just explained above. Not only the one-dimensional case, our techniques are equally available for the Keller-Segel equations in higher dimensional spaces (and also some other types of nonlinear parabolic equations \cite{AMY,IY}). The main difference of our result, Corollary \ref{C}, and \cite[Theorem 1.1]{FLP} is that our result gives an explicit convergence rate \eqref{30} at which the solution $(u(t),\rho(t))$ converges to a stationary solution $(\ol u,\ol\rho)$. A priori knowledge of such an order estimate of convergence is essentially meaningful in the stage of numerical computations for the limit solution $(\ol u,\ol\rho)$.
\par

%%%%  Notations  %%%%%%
Throughout the paper, $H^s(I)$ denotes the real Sobolev space in $I$ with exponent $s \ge 0$. As usual, $H^0(I) = L_2(I)$. The dual space of $H^s(I)$ is denoted by $H^s(I)'$. For $s > \frac32$, $H^s_N(I)$ denotes the closed subspace of $H^s(I)$ consisting of the functions $\rho \in H^s(I)$ satisfying the homogeneous Neumann conditions $\rho'(\alpha)=\rho'(\beta)=0$ at the boundary points of $I$.

%%%%%%%%%%%%%%%%%%%%%%%%%%%%%%%%%%%%%%%%%%%%%%%%%%%%%%%%%%%%%
%%%%%%%%%%%%%%%%%%%%%%%%%%%%%%%%%%%%%%%%%%%%%%%%%%%%%%%%%%%%%
%%%%%%%%%%%%%%%%%%%%%%%%%%%%%%%%%%%%%%%%%%%%%%%%%%%%%%%%%%%%%
\section{Reviews on Existence Results}

The existence of global solutions to \eqref{1}-\eqref{2} has already been obtained in \cite{OY}. In this section, let us briefly review the results and, at the same time, let us verify some properties of the global solutions which were not mentioned there but are necessary in this study.

%%%%%%%%%%%%%%%%%%%%%%%%%%%%%%%%%%%%%%%%%%%%%%%%%%%%%%%%%%%%%
\subsection{Global Solutions}

In order to construct the solution $(u,\rho)$ to \eqref{1}-\eqref{2}, we can apply the theory of semilinear abstract parabolic evolution equations (see \cite[Theorem 4.4]{Ya}). In fact, for any $(u_0,\rho_0) \in L_2(I) \times H^1(I)$ satisfying
\begin{equation} \label{6}
  u_0(x) \ge 0 \quad\text{for a.e. $x \in I$}\qquad\text{and}\qquad
  \rho_0(x) \ge 0 \quad\text{for $x \in I$},
\end{equation}
there exists a unique local solution lying in the function space:
\begin{equation*}
\left\{ \begin{aligned}
  0 \le u &\in \cl C([0,T_{(u_0,\rho_0)}]; L_2(I)) \cap
              \cl C((0,T_{(u_0,\rho_0)}]; H^2_N(I)) \cap
              \cl C^1((0,T_{(u_0,\rho_0)}]; L_2(I)),  \\
  0 \le \rho &\in \cl C([0,T_{(u_0,\rho_0)}]; H^1(I)) \cap
              \cl C((0,T_{(u_0,\rho_0)}]; H^3_N(I)) \cap
              \cl C^1((0,T_{(u_0,\rho_0)}]; H^1(I)).
\end{aligned} \right.
\end{equation*}
Here, the time $T_{(u_0,\rho_0)} > 0$ is determined by the magnitude of norms $\|u_0\|_{L_2}$ and $\|\rho_0\|_{H^1}$. In the meantime, as established by \cite[Proposition 4.1]{OY}, any local solution to \eqref{1}-\eqref{2} satisfies {\it a priori} estimates, which show that the norms $\|u(t)\|_{L_2}$ and $\|\rho(t)\|_{H^1}$ of any local solution are globally controlled by those of $u_0$ and $\rho_0$ alone. Hence, as stated in \cite[Theorem 4.3]{OY}, for any $(u_0,\rho_0) \in L_2(I) \times H^1(I)$ satisfying \eqref{6}, \eqref{1}-\eqref{2} possesses a unique global solution in the function space:
\begin{equation*}
\left\{ \begin{aligned}
  0 \le u &\in \cl C([0,\infty); L_2(I)) \cap
              \cl C((0,\infty); H^2_N(I)) \cap
              \cl C^1((0,\infty); L_2(I)),  \\
  0 \le \rho &\in \cl C([0,\infty); H^1(I)) \cap
              \cl C((0,\infty); H^3_N(I)) \cap
              \cl C^1((0,\infty); H^1(I)).
\end{aligned} \right.
\end{equation*}
\par

Moreover, if $(u_0,\rho_0)$ of \eqref{2} is taken as
\begin{equation} \label{6'}
  u_0 \in H^2_N(I) \quad\text{and}\quad  \rho_0 \in H^3_N(I),
\end{equation}
together with \eqref{6}, then the global solution $(u,\rho)$ belongs to
\begin{equation}  \label{10}
  u \in \cl C([0,\infty); H^2_N(I))
    \quad\text{and}\quad \rho \in \cl C([0,\infty); H^3_N(I)),
\end{equation}
respectively, and satisfies a global norm estimate
\begin{equation}  \label{11}
  \|u(t)\|_{H^2} + \|\rho(t)\|_{H^3} \le C_{(u_0,\rho_0)},
  \qquad 0 < t < \infty,
\end{equation}
with some constant $C_{(u_0,\rho_0)}$ depending only on $\|u_0\|_{H^2}$ and $\|\rho_0\|_{H^3}$. For the proof, see \cite[Theorem 4.4]{OY}.
\par

%%%%%%%%%%%%%%%%%%%%%%%%%%%%%%%%%%%%%%%%%%%%%%%%%%%%%%%%%%%%
\subsection{Strict Positivity of $u(t)$}

Let us next verify strict positivity of $u(t)$ which is very important in constructing a Lyapunov function. Let $(u_0,\rho_0)$ of \eqref{2} be taken as \eqref{6'} together with
\begin{equation} \label{9}
  u_0(x) > 0 \quad\text{for $x \in \ol I$}\quad\text{and}\quad
  \rho_0(x) \ge 0 \quad\text{for $x \in \ol I$}.
\end{equation}

\begin{proposition}  \label{P1}
Under \eqref{6'}, if $(u_0,\rho_0)$ satisfies \eqref{9}, then the solution $(u,\rho)$ of \eqref{1}-\eqref{2} also satisfies the same at every time $0 < t < \infty$.
\end{proposition}

\begin{proof}
As $\rho(t) \ge 0$ is already known, it suffices to estimate $u(t)$ from below. Put $\delta_0 = \min_{x \in \ol I} u_0(x) > 0$. And put $C_\tau = k\max_{(x,t) \in \ol I \times [0,\tau]} |\rho_{xx}(x,t)|$ for arbitrarily fixed $0 < \tau < \infty$. We regard $u$ as a solution to the linear diffusion equation
\begin{equation*}
  u_t = au_{xx} - p(x,t)u_x - q(x,t)u
  \qquad  \text{in}\quad  I \times (0,\tau),
\end{equation*}
where $p(x,t)=k\rho_x(x,t)$ and $q(x,t)=k\rho_{xx}(x,t)$.
\par

Using a cutoff function $H(\xi)$ such that $H(\xi) \equiv \frac12\xi^2$ for $-\infty < \xi < 0$ and $H(\xi) \equiv 0$ for $0 \le \xi < \infty$, consider the function
\begin{equation*}
  \phi(t) = \int_I H(u(x,t)-\delta_0 e^{-C_\tau t})dx,
  \qquad  0 \le t \le \tau.
\end{equation*}
Then,
\begin{align*}
  \frac{d\phi}{dt}(t) &= \int_I H'(u-\delta_0 e^{-C_\tau t})
          (u_t+\delta_0 C_\tau e^{-C_\tau t})dx   \\
      &= \int_I H'(u-\delta_0 e^{-C_\tau t})
          [au_{xx}-pu_x-qu+\delta_0 C_\tau e^{-C_\tau t}]dx.
\end{align*}
Here,
\begin{align*}
  a\int_I H'(u-\delta_0 e^{-C_\tau t})u_{xx}dx
  &= -a \int_I [H'(u-\delta_0 e^{-C_\tau t})]_x u_xdx  \\
  &= -a \int_I H''(u-\delta_0 e^{-C_\tau t})u_x^2dx \le 0.
\end{align*}
Meanwhile, noting that $p(\alpha,t)=p(\beta,t)=0$, we have
\begin{align*}
  -\int_I H'(u&-\delta_0 e^{-C_\tau t})pu_xdx
      = -\int_I H'(u-\delta_0 e^{-C_\tau t})p[u-\delta_0 e^{-C_\tau t}]_xdx \\
     &= \int_I [H''(u-\delta_0 e^{-C_\tau t})u_xp
            + H'(u-\delta_0 e^{-C_\tau t})p_x][u-\delta_0 e^{-C_\tau t}]dx.
\end{align*}
Furthermore, since $H''(\xi) \ge 0$ and $H'(\xi)\xi \ge 0$, it follows that
\begin{multline*}
  \int_I H''(u-\delta_0 e^{-C_\tau t})u_xp[u-\delta_0 e^{-C_\tau t}]dx  \\
   \le \frac a2 \int_I H''(u-\delta_0 e^{-C_\tau t})u_x^2dx
  + \frac1{2a}\int_I H''(u-\delta_0 e^{-C_\tau t})
      p^2[u-\delta_0 e^{-C_\tau t}]^2dx,
\end{multline*}
and that
\begin{equation*}
  \int_I H'(u-\delta_0 e^{-C_\tau t})p_x [u-\delta_0 e^{-C_\tau t}]dx
  \le \int_I H'(u-\delta_0 e^{-C_\tau t})|p_x|[u-\delta_0 e^{-C_\tau t}]dx.
\end{equation*}
Therefore, the relations $H''(\xi)\xi^2 = H'(\xi)\xi = 2H(\xi)$ give
\begin{equation*}
  -\int_I H'(u-\delta_0 e^{-C_\tau t})pu_xdx
  \le \frac a2 \int_I H''(u-\delta_0 e^{-C_\tau t})u_x^2dx 
          + D_\tau \varphi(t)
\end{equation*}
with some constant $D_\tau > 0$. Finally, since $u \ge 0$ and $-q+C_\tau \ge 0$, it is clear that
\begin{multline*}
  \int_I H'(u-\delta_0 e^{-C_\tau t})[-qu+\delta_0 C_\tau e^{-C_\tau t}]dx  \\
      = \int_I H'(u-\delta_0 e^{-C_\tau t})(-q+C_\tau)u\,dx
       - C_\tau \int_I H'(u-\delta_0 e^{-C_\tau t})[u-\delta_0 e^{-C_\tau t}]dx
      \le 0.
\end{multline*}
\par

We have thus shown that $\frac{d\phi}{dt}(t) \le D_\tau \phi(t)$ for any $0 \le t \le \tau$. Consequently, $\phi(t) \le \phi(0)e^{D_\tau t}$. Since $\phi(0)=0$, it follows that $\phi(t)$ vanishes identically on the interval $[0,\tau]$. In other words, $u(x,t) \ge \delta_0 e^{-C_\tau t}$ for $(x,t) \in \ol I \times [0,\tau]$. Since $\tau > 0$ was fixed arbitrarily, the desired strict positivity of $u(t)$ is obtained.
\end{proof}

%%%%%%%%%%%%%%%%%%%%%%%%%%%%%%%%%%%%%%%%%%%%%%%%%%%%%%%%%
\subsection{Lyapunov Function}

Let us further remember that the equation \eqref{1} admits a Lyapunov function. In fact, under \eqref{6'} and \eqref{9}, let $(u,\rho)$ denote the global solution of \eqref{1}-\eqref{2}. Of course, $(u,\rho)$ belongs to the function space \eqref{10} and satisfies the positivity \eqref{9} for $0 < t < \infty$. According to \cite[(7.2)]{OY}, it holds true for $(u,\rho)$ that
\begin{multline} \label{13}
   \frac d{dt} \int_I \left\{ac[u\log u-u]
      + \frac{bk}2\rho_x^2 + \frac{dk}2\rho^2 - cku\rho\right\}dx  \\
   = - c\int_I u\left\{[a\log u-k\rho]_x\right\}^2dx
     - k\int_I \rho_t^2dx \le 0,
   \qquad  0 \le t < \infty.
\end{multline}
This means that the function
\begin{equation}  \label{16}
  \Phi(u,\rho)
  = \int_I \left\{ac[u\log u- u]
      + \frac{bk}2(\rho')^2 + \frac{dk}2\rho^2 - cku\rho\right\}dx
\end{equation}
defined for $0 \le u \in L_2(I)$ and $\rho \in H^1(I)$ becomes a Lyapunov function for the solution $(u,\rho)$. (Setting $u\log u=0$ for $u=0$, too, let us consider $u\log u$ to be a continuous function for $u \in [0,\infty)$.)
\par

In addition, we notice the following fact.

\begin{proposition}  \label{P2}
If $\frac d{dt}\Phi(u(t),\rho(t)) = 0$ at some time $t=\ol t$, then $(u(\ol t),\rho(\ol t))$ is a stationary solution of \eqref{1}.
\end{proposition}

\begin{proof}
Assume that $\left[\frac d{dt}\Phi(u(t),\rho(t))\right]_{t=\ol t} = 0$ and put $\ol u=u(\ol t)$ and $\ol\rho=\rho(\ol t)$. Then, because of the strict positivity of $\ol u$ due to Proposition \ref{P1}, it must follow from \eqref{13} that
\begin{equation}  \label{24}
  [a\log\ol u-k\ol\rho]_x = 0            \qquad\text{and}\qquad
  b\ol\rho_{xx} - d\ol\rho + c\ol u = 0  \qquad\text{in}\enskip I.
\end{equation}
The first equality yields further
\begin{equation*}
  \left[a\ol u_x-k\ol u\,\ol\rho_x\right]_x
  = \left\{\ol u[a\log\ol u-k\ol\rho]_x\right\}_x = 0
    \qquad\text{in}\enskip I.
\end{equation*}
Hence, $(\ol u,\ol\rho)$ is a stationary solution to \eqref{1}.
\end{proof}

%%%%%%%%%%%%%%%%%%%%%%%%%%%%%%%%%%%%%%%%%%%%%%%%%%%%%%%%%%%%%
\subsection{$\omega$-limit set}

Under \eqref{6'} and \eqref{9}, let $(u,\rho)$ be the global solution of \eqref{1}-\eqref{2}. As usual, its $\omega$-limit set is defined by
\begin{equation*}
  \omega(u,\rho) = \{(\ol u,\ol\rho) \in L_2(I) \times L_2(I);\;
            \exists t_n \nearrow \infty, \enskip u(t_n) \to \ol u \enskip
            \text{and}\enskip \rho(t_n) \to \rho \enskip\text{in}
            \enskip L_2(I)\}.
\end{equation*}
As $(u,\rho)$ satisfies the global estimate \eqref{11}, it is clear that $\omega(u,\rho)\not=\emptyset$. Furthermore, as any bounded, closed ball of $H^2(I)$ (resp. $H^3(I)$) is weakly sequentially closed in $H^2(I)$ (resp. $H^3(I)$), it follows that $\omega(u,\rho) \subset H^2_N(I) \times H^3_N(I)$. In addition, if $(u(t_n),\rho(t_n)) \to (\ol u,\ol\rho)$ in $L_2(I) \times L_2(I)$, then this sequence is convergent even in $H^{2\theta}(I)\times H^{3\theta}(I)$ for any exponent $0 < \theta < 1$; in particular, $(\ol u,\ol\rho)$ satisfies \eqref{6}. As an immediate consequence of these facts and Proposition \ref{P2}, we verify that
\begin{equation}  \label{13*}
  \lim_{t\to\infty} \Phi(u(t),\rho(t))
  = \inf_{0 \le t < \infty} \Phi(u(t),\rho(t))
  = \Phi(\ol u,\ol\rho)
    \qquad\text{for any}\enskip (\ol u,\ol\rho) \in \omega(u,\rho).
\end{equation}
\par

In the meantime, according to \cite[Theorem 7.1]{OY}, $\omega(u,\rho)$ must contain at least one stationary solution of \eqref{1}, say $(\ol u,\ol\rho)$. Then,  $\ol u$ is verified to be strictly positive on $\ol I$.

\begin{proposition}  \label{T2}
Let $(\ol u,\ol\rho) \in \omega(u,\rho)$ be a stationary solution of \eqref{1}. Then, $(\ol u,\ol\rho)$ must satisfy \eqref{9}.
\end{proposition}

\begin{proof}
Let us regard $\ol u$ as a solution to the ordinary differential equation $a\ol u'' - p(x)\ol u' - q(x)\ol u = 0$ for $x \in I$, where $p(x)=k\ol\rho'(x)$ and $q(x)=k\ol\rho''(x)$. Because of $(\ol u,\ol\rho) \in H^2(I)\times H^3(I)$, $\ol u$ becomes a classical solution. So, if $\ol u(x_0)=0$ at some $x_0 \in I$, then $\ol u(x) \ge 0$ on $\ol I$ implies that $\ol u'(x_0)=0$; consequently, $\ol u$ must vanish identically in $I$. But this contradicts to the condition that $\int_I \ol u(x)dx=\|u_0\|_{L_1}>0$. Similarly, if $\ol u(x_0)=0$ at $x_0=\alpha$ or $\beta$, then, since the Neumann boundary conditions imply $\ol u'(x_0)=0$, the same contradiction takes place. Hence, $\ol u(x) > 0$ for $x \in \ol I$.
\par

Meanwhile, $\rho(x,t) \ge 0$ implies that $\ol\rho(x) \ge 0$ for $x \in \ol I$.
\end{proof}

%%%%%%%%%%%%%%%%%%%%%%%%%%%%%%%%%%%%%%%%%%%%%%%%%%%%%%%%%%%%%%%%%%
%%%%%%%%%%%%%%%%%%%%%%%%%%%%%%%%%%%%%%%%%%%%%%%%%%%%%%%%%%%%%%%%%%
%%%%%%%%%%%%%%%%%%%%%%%%%%%%%%%%%%%%%%%%%%%%%%%%%%%%%%%%%%%%%%%%%%
\section{Formulation}

In what follows, we will arbitrarily fix an initial function $(u_0,\rho_0)$ satisfying \eqref{6'} and \eqref{9}. Let $(u,\rho)$ be the global solution of \eqref{1}-\eqref{2}. Let $(\ol u,\ol\rho)$ be an $\omega$-limit of $(u,\rho)$ which is a stationary solution of \eqref{1}. Then, our goal is to show that $\omega(u,\rho)$ is a singleton, i.e., $\omega(u,\rho)=\{(\ol u,\ol\rho)\}$, and consequently $(u(t),\rho(t))$ converges asymptotically to $(\ol u,\ol\rho)$.
\par

For establishing this convergence, we will make use of the methods and techniques of Functional Analysis developed by \cite{AMY,Ch,HJ,Je}. This section is then devoted to setting up their framework.

%%%%%%%%%%%%%%%%%%%%%%%%%%%%%%%%%%%%%%%%%%%%%%%%%%%%%%%%%%%%%%%%%%
\subsection{Evolution Equation}

It follows from the equation of $u$ of \eqref{1} that $\frac d{dt}\int_I u(x,t)dx = 0$ for every $0 < t < \infty$; therefore, $\int_I u(x,t)dx \equiv \|u_0\|_{L_1}$. In view of this fact, we want to shift the unknown function $u(x,t)$ to $v(x,t) = u(x,t)-f$, where $f$ stands for the positive constant $\|u_0\|_{L_1}|I|^{-1}$. Then, $v$ satisfies $\int_I v(x,t)dx = 0$ together with the equation
\begin{equation}  \label{3}
\left\{\begin{aligned}
  &v_t = av_{xx} - k[(v+f)\rho_x]_x
       \hspace{1.5cm}\text{in}\quad I\times(0,\infty),   \\  
  &v_x = 0
       \hspace{4.8cm}\text{on}\quad \partial I\times(0,\infty),
\end{aligned}\right.
\end{equation}
\par

In view of this shifting, we are led to introduce the subspace
\begin{equation*}
  L_{2,m}(I) = \left\{v \in L_2(I);\; \int_I v(x)dx = 0\right\}
\end{equation*}
of $L_2(I)$ (consisting of zero-mean functions in $I$) and to formulate the equation \eqref{3} of $v$ in this subspace. Clearly, $L_{2,m}(I)$ is an orthogonal compliment of the $1$-dimensional subspace of constant functions in $I$ and its orthogonal projection $P_m$ is given by
\begin{equation} \label{25}
\left\{ \begin{aligned}
  P_mu &= u - \frac1{|I|}\int_I u(x)dx, \qquad u \in L_2(I),  \\
  Q_mu &= (1-P_m)u = \frac1{|I|}\int_I u(x)dx,
          \qquad  u \in L_2(I).
\end{aligned} \right.
\end{equation}
\par

The realization of $-a\frac{d^2}{dx^2}$ in $L_{2,m}(I)$ under the homogeneous Neumann boundary conditions is defined as follows. Consider a bilinear form
\begin{equation*}
  a_1(v,w) = a\int_I v'(x)w'(x)dx,
  \qquad v,\, w \in H^1_m(I),
\end{equation*}
on $H^1_m(I) = \{v \in H^1(I);\; Q_mv = 0\}$. Since $H^1_m(I)$ can be equipped with an inner product
\begin{equation}  \label{7}
  (v,w)_{H^1_m} = a_1(v,w), \qquad  v,\, w \in H^1_m(I),
\end{equation}
the form $a_1(v,w)$ is trivially coercive. Therefore, by the relation
\begin{equation*}
  a_1(v,w)=\left<\cl A_1v,w\right>_{H^1_m{}'\times H^1_m},
  \qquad  v,\, w \in H^1_m(I),
\end{equation*}
an isomorphism $\cl A_1$ from $H^1_m(I)$ onto $H^1_m(I)'$ is determined (see Dautray-Lions \cite[Chapter VII]{DL}), where $H^1_m(I)'$ is the adjoint space of $H^1_m(I)$ such that $H^1_m(I) \subset L_{2,m}(I) \subset H^1_m(I)'$ makes a triplet of spaces. On the other hand, according to \cite[Theorem 6.1]{Ya1}, $\cl A_1$ is a real sectorial operator of $H^1_m(I)'$. As usual, $\cl A_1$ is regarded as a realization of $-a\frac{d^2}{dx^2}$ in $H^1_m(I)'$ under the homogeneous Neumann boundary conditions and its part $A_1$ in $L_{2,m}(I)$ is regarded as a realization of the same in $L_{2,m}(I)$. 
\par

Meanwhile, by virtue of the lemma below, we observe that $-k[(v+f)\rho']' \in H^1_m(I)'$ if $v \in H^1_m(I)$ and $\rho \in H^2_N(I)$. Hence, if we set a nonlinear operator $\chi(v,\rho)=-k[(v+f)\rho']'$ acting from $H^1_m(I) \times H^2_N(I)$ into $H^1_m(I)'$, then the equation \eqref{3} can be written as $\frac{dv}{dt} + \cl A_1v = \chi(v,\rho)$ in $H^1_m(I)'$.

\begin{lemma}  \label{L1}
If $\eta \in H^1(I)$ satisfies $\eta(\alpha)=\eta(\beta)=0$, then $\eta' \in L_{2,m}(I)$ and it holds that $\|\eta'\|_{{H^1_m}'} \le (1/\sqrt{a})\|\eta\|_{L_2}$.
\end{lemma}

\begin{proof}
The boundary conditions imply that $\int_I \eta'(x)dx=0$, i.e., $\eta' \in L_{2,m}(I)$. Then,
\begin{equation*}
  |\left<\eta',v\right>_{H^1_m{}'\times H^1_m}|
  = |(\eta',v)_{L_{2,m}}| = |(\eta,v')_{L_2}|
  \le (1/\sqrt{a})\|\eta\|_{L_2}\|v\|_{H^1_m},
  \quad  v \in H^1_m(I),
\end{equation*}
due to \eqref{7}. Hence, the result is verified.
\end{proof}

As for the equation of $\rho$, we will formulate it in $L_2(I)$. Consider a bilinear form
\begin{equation*}
  a_2(\rho,\varphi) = b\int_I \rho'\varphi'\,dx + d\int_I \rho\varphi\,dx,
  \qquad \rho,\, \varphi \in H^1(I).
\end{equation*}
As $H^1(I)$ can be equipped with an inner product
\begin{equation} \label{4}
  (\rho,\varphi)_{H^1} = a_2(\rho,\varphi),
  \qquad  \rho,\,\varphi \in H^1(I),
\end{equation}
the form $a_2(\rho,\varphi)$ is trivially coercive. Then, by the same methods as for $a_1(v,w)$, $a_2(\rho,\varphi)$ determines an isomorphism $\cl A_2$ from $H^1(I)$ onto $H^1(I)'$ which is a realization of the differential operator $-b\frac{d^2}{dx^2}+d$ in $H^1(I)'$ under the homogeneous Neumann boundary conditions. Let $A_2$ be its part in $L_2(I)$ which can be characterized by the relation
\begin{equation*}
  a_2(\rho,\varphi) = (A_2\rho,\varphi)_{L_2},
  \qquad  \rho \in \cl D(A_2),\; \varphi \in L_2(I).
\end{equation*}
Then, $A_2$ is seen to be a positive definite self-adjoint operator of $L_2(I)$ with domain $\cl D(A_2) = H^2_N(I)$ whose square root $A_2^\frac12$ has the domain $\cl D(A_2^\frac12) = H^1(I)$ together with
\begin{equation} \label{23}
  a_2(\rho,\varphi) = (A_2^\frac12\rho,A_2^\frac12\varphi)_{L_2},
  \qquad  \rho,\,\varphi \in \cl D(A_2^\frac12).
\end{equation}
Consequently, the spaces $H^2_N(I) \subset H^1(I) \subset L_2(I)$ make a triplet. Hence, the equation of \eqref{1} for $\rho$ is written as $\frac{d\rho}{dt} + A_2\rho = c(v+f)$ in $L_2(I)$.
\par

We have thus arrived at the following formulation of \eqref{1}. Set an underlying space
\begin{equation}  \label{8}
  \frak X = \left\{\begin{pmatrix} w  \\  \varphi \end{pmatrix};\;
      w \in  H^1_m(I)' \enskip\text{and}\enskip
      \varphi \in L_2(I) \right\}.
\end{equation}
Let $A = \begin{pmatrix}  \cl A_1  &  0  \\  0  &  A_2 \end{pmatrix}$ be an operator matrix acting in $\frak X$ with the domain
\begin{equation}  \label{5}
  \cl D(A) = \left\{\begin{pmatrix}  v  \\  \rho \end{pmatrix};\;
      v \in \cl D(\cl A_1) = H^1_m(I)  \enskip\text{and}\enskip
      \rho \in \cl D(A_2) = H^2_N(I) \right\}.
\end{equation}
As seen, $A$ is a real sectorial operator of $\frak X$. Meanwhile, let $F(V)$ be a nonlinear operator of $\frak X$ defined by
\begin{equation*}
  F(V) = \begin{pmatrix} -k[(v+f)\rho']'  \\
             c(v+f)  \end{pmatrix},
  \qquad V = \begin{pmatrix} v \\ \rho \end{pmatrix} \in \cl D(A).
\end{equation*}
Then, the function $V(t) = \begin{pmatrix}v(t) \\ \rho(t)\end{pmatrix}$, where $v(t) = u(t)-f$, is considered to be a solution to the evolution equation
\begin{equation}  \label{39}
  \tfrac{dV}{dt} + AV = F(V),  \qquad 0 < t < \infty,
\end{equation}
in $\frak X$ satisfying the initial condition
\begin{equation} \label{40}
  V(0) = V_0 \equiv \begin{pmatrix} u_0-f \\ \rho_0 \end{pmatrix}. 
\end{equation}

%%%%%%%%%%%%%%%%%%%%%%%%%%%%%%%%%%%%%%%%%%%%%%%%%%%%%%%%%%%%%%%%%%%%%%%
\subsection{Triplet of Spaces}

Let us now set up a triplet of spaces which give the framework of our arguments. Taking account of \eqref{16} and \eqref{39}, we are led to introduce the following triplet $Z \subset X \subset Z^*$. The space $Z^*$ is taken as the underlying of the evolution equation \eqref{39}, i.e., $Z^* = \frak X$. The space $X$ must be taken as the definition space of the Lyapunov function \eqref{16}, namely, $X$ is set as
\begin{equation}  \label{17}
  X = \left\{\begin{pmatrix} v \\ \rho  \end{pmatrix} ;\;
    v \in L_{2,m}(I) \enskip\text{and}\enskip \rho \in H^1(I) \right\}.
\end{equation}
Then, $Z$ is naturally taken as $Z = \cl D(A)$. Thereby, its duality product is given by
\begin{equation*}
  \left<\begin{pmatrix}v \\ \rho\end{pmatrix},
        \begin{pmatrix}w \\ \varphi\end{pmatrix}\right>_{Z \times Z^*}
  = \left<v,w\right>_{H^1_m \times H^1_m{}'}
    + (A_2\rho,\varphi)_{L_2},
    \qquad  \begin{pmatrix}v \\ \rho\end{pmatrix} \in Z,\;
           \begin{pmatrix}w \\ \varphi\end{pmatrix} \in Z^*.
\end{equation*}
As a general property of triplet, it is known that
\begin{equation} \label{12}
  \|V\|_X \le \|V\|_Z^\frac12 \|V\|_{Z^*}^\frac12,
      \qquad V \in Z.
\end{equation}
\par

By \eqref{10} and \eqref{11}, the solution $V$ of \eqref{39}-\eqref{40} certainly belongs to the space
\begin{equation}  \label{26}
  V \in \cl C([0,\infty); Z) \cap \cl C^1([0,\infty); Z^*)
\end{equation}
and satisfies the global norm estimate
\begin{equation}  \label{14}
  \|V(t)\|_Z + \|\tfrac{dV}{dt}(t)\|_{Z^*} \le R, \qquad 0 \le t < \infty,
\end{equation}
$R$ being some constant.

%%%%%%%%%%%%%%%%%%%%%%%%%%%%%%%%%%%%%%%%%%%%%%%%%%%%%%%%%%%%%%%%%
\subsection{Lyapunov Function}

From \eqref{16}, the Lyapunov function is rewritten as
\begin{multline*}
  \Phi(V)
  = \int_I \bigg\{ac[(v+f)\log(v+f)-(v+f)] + \frac{bk}2(\rho')^2   \\
      + \frac{dk}2\rho^2 - ck(v+f)\rho\bigg\}dx,
      \qquad  V = \begin{pmatrix} v \\ \rho \end{pmatrix},
\end{multline*}
which is defined for $-f \le v \in L_{2,m}(I)$ and $\rho \in H^1(I)$. The equality \eqref{13} is rewritten as
\begin{equation}  \label{13'}
   \frac d{dt} \Phi(V(t))
   = - c\int_I (v+f)\{[a\log(v+f)-k\rho]_x\}^2dx
     - k\int_I (\rho_t)^2dx \le 0.
\end{equation}
\par

According as Proposition \ref{P2}, we can claim that
\begin{equation}  \label{18}
  \text{if $\tfrac{d}{dt} \Phi(V(t)) = 0$ at some $t = \ol t$, then $\ol V=V(\ol t)$ is a stationary solution of \eqref{39}.}
\end{equation}

%%%%%%%%%%%%%%%%%%%%%%%%%%%%%%%%%%%%%%%%%%%%%%%%%%%%%%%%%%%%%%%%%%%
\subsection{$\omega$-Limit of $V(t)$}

For the solution $V(t)$, let us redefine its $\omega$-limit set as
\begin{equation}  \label{35}
  \omega(V) = \{\ol V \in X;\; \exists t_n
            \nearrow \infty, \enskip V(t_n) \to \ol V \enskip\text{in}
            \enskip X\}.
\end{equation}
It is clear that $\emptyset \not= \omega(V) \subset \ol B^Z(0; R)$. As seen by \eqref{13*}, it holds for $V(t)$ that
\begin{equation} \label{13*'}
  \lim_{t \to \infty} \Phi(V(t)) = \inf_{0 \le t < \infty} \Phi(V(t))
  = \Phi(\ol V)
  \qquad \text{for any}\enskip  \ol V \in \omega(V).
\end{equation}
According to Proposition \ref{T2}, there is an $\omega$-limit $\ol V={}^t(\ol v,\ol\rho)$ which is a stationary solution of \eqref{39} and $(\ol v+f,\ol\rho)$ satisfies the positivity \eqref{9}. In view of this fact, we fix a constant $\delta > 0$ for which it holds that
\begin{equation}  \label{34}
  \ol v(x)+f \ge \delta \quad\text{for every}\enskip x \in \ol I.
\end{equation}
\par

Our goal is thereby to show that $V(t)$ converges asymptotically to $\ol V$.

%%%%%%%%%%%%%%%%%%%%%%%%%%%%%%%%%%%%%%%%%%%%%%%%%%%%%%%%%%%%%%%%%%
\subsection{Suitable Extension of $\Phi(V)$}

As a matter of fact, we need values of the Lyapunov function $\Phi(V)$ only for vectors $V$ lying in some neighborhood of the $\omega$-limit $\ol V$. By this reason, it is very convenient to extend the definition space of $\Phi(V)$ to the whole $X$.
\par

In view of \eqref{34}, we extend $\log\xi$ to a smooth function $\log^\sim\xi$ in the whole line $(-\infty,\infty)$ in such a way that
\begin{equation*}
  \log^\sim\xi \equiv 0 \quad\text{for $-\infty < \xi \le 0$},\qquad
  \log^\sim\xi \equiv \log \xi \quad
       \text{for $\tfrac\delta2 < \xi < \infty$},
\end{equation*}
and the values $\log^\sim\xi$ being suitably defined for $0 < \xi \le \frac\delta2$ and put
\begin{equation}  \label{22}
  \ell(\xi) = \xi\log^\sim\xi\; -\; \xi
  \qquad\text{for}\enskip  -\infty < \xi < \infty.
\end{equation}
Using this function, we extend $\Phi(V)$ as
\begin{equation*}
  \wt\Phi(V)
  = \int_I \left\{ac\ell(v+f)
      + \frac{bk}2(\rho')^2
      + \frac{dk}2\rho^2 - ck(v+f)\rho\right\}dx,
      \qquad   V = \begin{pmatrix} v \\ \rho \end{pmatrix} \in X.
\end{equation*}
\par

The following properties of $\wt\Phi(V)$ are then verified.

\begin{proposition}  \label{P3}
The function $\wt\Phi\,{:}\, X \to \bb R$ is continuously Fr\'echet differentiable on $X$. Identifying $X$ with its adjoint $X'$, define the derivative $\wt\Phi'(V) \in X$ by the formula
\begin{equation}  \label{36}
  |\wt\Phi(V+H) - \wt\Phi(V) - (\wt\Phi'(V),H)_X| = o(\|H\|_X)
  \quad\text{as $H \to 0$ in $X$}.
\end{equation}
Then, $\wt\Phi'(V)$ is given by
\begin{equation}  \label{21}
  \wt\Phi'(V) = \begin{pmatrix} cP_m[a\ell'(v+f)-k\rho]  \\
                k[\rho-cA_2^{-1}(v+f)] \end{pmatrix},
  \qquad V = \begin{pmatrix} v \\ \rho \end{pmatrix} \in X.
\end{equation}
\end{proposition}

\begin{proof}
For $v,\, h \in L_{2,m}(I)$, we have
\begin{align*}
  \int_I [\ell(v+h+f)&-\ell(v+f)-\ell'(v+f)h]dx  \\
  &= \int_I \left[\int_0^1 \ell'(v+\theta h+f)hd\theta-\ell'(v+f)h\right]dx  \\
  &= \int_I \int_0^1 \left[\ell'(v+\theta h+f)-\ell'(v+f)\right]h\,d\theta dx.
\end{align*}
In view of \eqref{22}, it is seen that $\sup_{-\infty<\xi<\infty} |\ell''(\xi)|<\infty$. Therefore,
\begin{equation*}
  \left|\int_I [\ell(v+h+f)-\ell(v+h)-\ell'(v+f)h]dx\right|
  \le C\|h\|_{L_{2,m}}^2.
\end{equation*}
This means that the function $v \mapsto \int_I \ell(v+f)dx$ is Fr\'echet differentiable in $L_{2,m}(I)$ and its derivative at $v$ is given by $h \mapsto \int_I \ell'(v+f)h\,dx = (P_m\ell'(v+f),h)_{L_{2,m}}$.
\par

Next, for $\rho,\, \eta \in H^1(I)$, we have
\begin{equation*}
  \int_I \left[\tfrac b2(\rho'+\eta')^2 + \tfrac d2(\rho+\eta)^2
             - \tfrac b2(\rho')^2 - \tfrac d2\rho^2   
             - b\rho'\eta' - d\rho\eta\right]dx
    = \int_I [\tfrac b2(\eta')^2+\tfrac d2\eta^2]dx.
\end{equation*}
Therefore,
\begin{equation*}
  \left|\int_I \left[\tfrac b2(\rho'+\eta')^2 + \tfrac d2(\rho+\eta)^2
             - \tfrac b2(\rho')^2 - \tfrac d2\rho^2\right]dx
             - a_2(\rho,\eta)\right|
    \le C\|\eta\|_{H^1}^2.
\end{equation*}
This means that the function $\rho \mapsto \int_I [\frac b2(\rho')^2+\frac d2\rho^2]dx$ is Fr\'echet differentiable in $H^1(I)$ and its derivative at $\rho$ is given by $\eta \mapsto a_2(\rho,\eta) = (\rho,\eta)_{H^1}$ due to \eqref{4}.
\par

Finally, we have
\begin{equation*}
   \int_I [(v+h+f)(\rho+\eta) - (v+f)\rho -(v+f)\eta - h\rho]dx
   = \int_I h\eta\,dx.
\end{equation*}
Here, since
\begin{equation*}
  \int_I (v+f)\eta\, dx
  = (A_2^\frac12A_2^{-1}(v+f),A_2^\frac12\eta)_{L_2}
  = (A_2^{-1}(v+f),\eta)_{H^1}
\end{equation*}
due to \eqref{4} and \eqref{23} and since $\int_I h\rho\,dx = (h,P_m\rho)_{L_{2,m}}$, it is seen that the function $V \mapsto \int_I (v+f)\rho\,dx$ is Fr\'echet differentiable in $X$ and its derivative at $V$ is given by $H = {}^t(h,\eta) \mapsto (P_m\rho,h)_{L_{2,m}} + (A_2^{-1}(v+f),\eta)_{H^1}$.
\par

We have thus proved that $\wt\Phi(V)$ is differentiable in the sense of \eqref{36} and its derivative is given by \eqref{21}. It is easy to verify that $V \mapsto \wt\Phi'(V)$ is continuous from $X$ into itself.
\end{proof}

\begin{proposition}  \label{P4}
Let $\ol V$ be the $\omega$-limit of $V(t)$ fixed in Subsection 3.4. Then, $\ol V$ satisfies $\wt\Phi'(\ol V)=0$, that is, $\ol V$ is a critical point of $\wt\Phi(V)$.
\end{proposition}

\begin{proof}
Consider the constant functions $\ol u(t) \equiv \ol u = \ol v+f,\, 0 \le t < \infty,$ and $\ol\rho(t) \equiv \ol\rho,\, 0 \le t < \infty$. Since $(\ol u(\cdot),\, \ol\rho(\cdot))$ is a stationary solution of \eqref{1}, Proposition \ref{P2} is available. As a consequence, the equalities in \eqref{24} hold true, i.e.,
\begin{equation*}
  a\log(\ol v+f)-k\ol\rho={\rm const.}
  \qquad\text{and}\qquad  b\ol\rho''-d\ol\rho+c(\ol v+f)=0.
\end{equation*}
Hence, $P_m[a\log(\ol v+f)-k\ol\rho]=0$ due to \eqref{25} and $A_2\ol\rho-c(\ol v+f)=0$. Noticing that $\ell'(\ol v+f) = \log(\ol v+f)$ by \eqref{34} and \eqref{22}, we observe that $\wt\Phi'(\ol V)=0$.
\end{proof}

\begin{proposition}
If $V \in Z$, then its derivative $\wt\Phi'(V)$ must also be in $Z$. Moreover, $\wt\Phi'\,{:}\, Z \to Z$ is continuous from $Z$ into itself.
\end{proposition}

\begin{proof}
It is easy to see that the mapping $v \mapsto \ell'(v)$ is continuous from $H^1_m(I)$ into $H^1(I)$. Meanwhile, from \eqref{25}, $P_m$ is seen to be a continuous projection from $H^1(I)$ onto $H^1_m(I)$, too. Therefore, $V \in Z$ implies $P_m[a\ell'(v+f)-k\rho] \in H^1_m(I)$. Meanwhile, it is clear that $V \in Z$ implies $\rho-cA_2^{-1}(v+f) \in \cl D(A_2)$. Hence, $V \in Z$ implies $\wt\Phi'(V) \in Z$.
\par

It is also immediate to see that $V \mapsto \wt\Phi'(V)$ is continuous from $Z$ into itself.
\end{proof}

As an immediate consequence of this proposition and \eqref{26}, we have
\begin{equation}  \label{37}
  \tfrac d{dt}\wt\Phi(V(t))
      = \left<\wt\Phi'(V(t)),\tfrac{dV}{dt}(t)\right>_{Z \times Z^*},
      \qquad  0 < t < \infty.
\end{equation}

%%%%%%%%%%%%%%%%%%%%%%%%%%%%%%%%%%%%%%%%%%%%%%%%%%%%%%%%%%%%%%%%%%%%%%
%%%%%%%%%%%%%%%%%%%%%%%%%%%%%%%%%%%%%%%%%%%%%%%%%%%%%%%%%%%%%%%%%%%%%%
%%%%%%%%%%%%%%%%%%%%%%%%%%%%%%%%%%%%%%%%%%%%%%%%%%%%%%%%%%%%%%%%%%%%%%
\section{Asymptotic Convergence of $V(t)$}

It is now possible to state the four fundamental properties which the solution $V(t)$, its $\omega$-limit $\ol V$ and the Lyapunov function $\wt\Phi(V)$ enjoy in the framework of $Z \subset X \subset Z^*$.
\smallskip

\noindent
(I) {\sl Critical Condition.} The $\omega$-limit $\ol V$ is a critical point of $\wt\Phi(V)$, i.e., $\wt\Phi'(\ol V)=0$.
\smallskip

\noindent
(II) {\sl Lyapunov Function.} There is a radius $r' > 0$ such that
\begin{equation}  \label{15}
  \wt\Phi(V(t)) > \wt\Phi(\ol V) \quad\text{and}\quad
  \tfrac d{dt} \wt\Phi(V(t)) \le 0
  \qquad\text{if}\enskip V(t) \in B^X(\ol V; r').
\end{equation}
\smallskip

\noindent
(III) {\sl Angle Condition.}
There exist a constant $\ve > 0$ and a radius $r''>0$ such that
\begin{equation}  \label{19}
  -\left<\wt\Phi'(V(t)),\tfrac{dV}{dt}(t)\right>_{Z \times Z^*}
  \ge \ve \|\wt\Phi'(V(t))\|_Z \|\tfrac{dV}{dt}(t)\|_{Z^*}
  \qquad\text{if}\enskip  V(t) \in B^X(\ol V; r'').
\end{equation}
\smallskip

\noindent
(IV) {\sl Gradient Inequality.}
There exist an exponent $0 < \theta \le \frac12$, a radius $r''' > 0$ and a constant $D > 0$ for which it holds true that
\begin{equation}  \label{28}
  -\|\wt\Phi'(V(t))\|_Z
  \ge D |\wt\Phi(V(t)) - \wt\Phi(\ol V)|^{1-\theta}
  \qquad\text{if}\enskip V(t) \in B^X(\ol V; r''').
\end{equation}
This inequality is called the {\L}ojasiewicz-Simon gradient inequality of $\wt\Phi(V)$.
\smallskip

Critical Condition was already verified by Proposition \ref{P4}.
\par

Also the two conditions of Lyapunov Function were essentially verified by \eqref{13'} and \eqref{18}. Indeed, as $H^{\frac34}(I) \subset \cl C(\ol I)$, there is a radius $r'>0$ such that $B^{L_2}(\ol v; r') \cap \ol B^{H^1}(0; R) \subset B^{\cl C}(\ol v; \frac\delta2)$, $R$ being the constant in \eqref{14}. By definition, it then follows that
\begin{equation}  \label{33}
   \wt\Phi(V) = \Phi(V)
   \qquad\text{if}\enskip V \in B^X(\ol V; r') \cap \ol B^Z(0; R).
\end{equation}
Therefore, in view of \eqref{14}, we observe from \eqref{13'} that $\frac{d}{dt}\wt\Phi(V(t)) = \frac{d}{dt}\Phi(V(t)) \le 0$ if $V(t) \in B^X(\ol V; r')$. Meanwhile, as mentioned in \eqref{18}, if $\frac{d}{dt}\Phi(V(t))=0$ at some $t=\ol t$, then $V(\ol t)$ is a stationary solution of \eqref{39} and the assertion of Theorem \ref{T3} to be proved is automatically valid. Therefore, it suffices to argue under the condition that $\frac{d}{dt}\Phi(V(t)) < 0$ for every $0 < t < \infty$, namely, $\Phi(V(t)) > \Phi(\ol V)$ for every $0 \le t < \infty$ due to \eqref{13*'}.
\par

To the contrary, in order to verify Angle Condition and Gradient Inequality, we have to investigate much deeper properties of $\wt\Phi(V)$. So, their verification will be described in the subsequent two sections. Before that, we here want to state and prove our main result which is directly derived by using these four properties only.

\begin{theorem} \label{T3}
As $t \to \infty$, the solution $V(t)$ of \eqref{39}-\eqref{40} converges to $\ol V$ in $Z^*$ at the rate
\begin{equation}  \label{30}
  \|V(t)-\ol V\|_{Z^*}
     \le (D\ve\theta)^{-1} [\Phi(V(t))-\Phi(\ol V)]^\theta
  \qquad\text{for all sufficiently large $t$}.
\end{equation}
\end{theorem}

This theorem is actually proved as an immediate consequence of the following crucial proposition.

\begin{proposition}
Put $r = \min \{r', r'', r'''\}$ and let $0 \le s < t < \infty$ be such that, for every $\tau \in [s,t]$, the value $V(\tau)$ stays in $B^X(\ol V; r)$. Then, we must have
\begin{equation}  \label{27}
  \left\|V(t)-V(s)\right\|_{Z^*}
  \le (D\ve\theta)^{-1}[\wt\Phi(V(s))-\wt\Phi(\ol V)]^\theta.
\end{equation} 
\end{proposition}

\begin{proof}
Since $\wt\Phi(V(\tau)) > \wt\Phi(\ol V)$ for $\tau \in [s,t]$ due to \eqref{15}, $-[\wt\Phi(V(\tau))-\wt\Phi(\ol V)]^\theta$ is differentiable in $\tau$. Furthermore, \eqref{37}, \eqref{19} and \eqref{28} yield the estimate
\begin{multline*}
  -\tfrac{d}{d\tau}[\wt\Phi(V(\tau))-\wt\Phi(\ol V)]^\theta
   = -\theta[\wt\Phi(V(\tau))-\wt\Phi(\ol V)]^{\theta-1}
               \tfrac{d}{d\tau}\wt\Phi(V(\tau))   \\
   \ge \ve\theta[\wt\Phi(V(\tau)-\wt\Phi(\ol V)]^{\theta-1}
                \|\wt\Phi'(V(\tau))\|_Z
                \left\|\tfrac{dV}{d\tau}(\tau)\right\|_{Z^*}
   \ge D\ve\theta\left\|\tfrac{dV}{d\tau}(\tau)\right\|_{Z^*}.
\end{multline*}
Integration of this inequality on $[s,t]$ gives
\begin{align*}
  [\wt\Phi(V(s))-\wt\Phi(\ol V)]^\theta
      - [\wt\Phi(V(t))-\wt\Phi(\ol V)]^\theta
  &\ge D\ve\theta\textstyle\int_s^t \left\|\tfrac{dV}{d\tau}(\tau)
                \right\|_{Z^*}d\tau   \\
  &\ge D\ve\theta\, \|V(t)-V(s)\|_{Z^*},
\end{align*}
which shows that \eqref{27} is valid for the $s$ and $t$.
\end{proof}

\begin{proof}[Proof of Theorem \ref{T3}]
Let $0 \le s < t < \infty$ such that $V(\tau) \in B^X(\ol V; r)$ for every $\tau \in [s,t]$. On account of \eqref{12} and \eqref{14}, it follows from \eqref{27} that
\begin{multline}  \label{29}
  \|V(t)-V(s)\|_X \le \|V(t)-V(s)\|_Z^\frac12 \|V(t)-V(s)\|_{Z^*}^\frac12  \\
    \le [2R(D\ve\theta)^{-1}]^\frac12
        [\wt\Phi(V(s))-\wt\Phi(\ol V)]^\frac\theta2.
\end{multline}
\par

Consider a temporal sequence $t_n \nearrow \infty$ such that $V(t_n) \to \ol V$ in $X$. Then, by \eqref{35} and \eqref{33}, there exists some $t_N$ such that
\begin{equation*}
  \|V(t_N)-\ol V\|_X \le \tfrac r3  \qquad\text{and}\qquad
  [2R(D\ve\theta)^{-1}]^\frac12
  [\wt\Phi(V(t_N))-\wt\Phi(\ol V)]^\frac\theta2 \le \tfrac r3.
\end{equation*}
Therefore, for any $t \ge t_N$ such that $V(\tau) \in B^X(\ol V; r)$ holds for every $\tau \in [t_N,t]$, \eqref{29} (with $s=t_N$) implies that
\begin{equation*}
  \|V(t)-\ol V\|_X \le \|V(t)-V(t_N)\|_X + \|V(t_N)-\ol V\|_X \le \tfrac{2r}3,
\end{equation*}
which shows an important fact that, after the $t_N$, all the values $V(t)$ stay in the ball $B^X(\ol V; \frac{2r}3) (\subset B^X(\ol V; r))$.
\par

Let $s$ be such that $t_N \le s < \infty$. For any $t_n \ge s$, apply \eqref{27} with $t=t_n$, then
\begin{equation*}
  \|V(t_n)-V(s)\|_{Z^*}
     \le (D\ve\theta)^{-1}[\wt\Phi(V(s))-\wt\Phi(\ol V)]^\theta.
\end{equation*}
Letting $n \to \infty$, we arrive at $\|\ol V-V(s)\|_{Z^*} \le (D\ve\theta)^{-1}[\wt\Phi(V(s))-\wt\Phi(\ol V)]^\theta$.
\par

On account of \eqref{33}, we conclude that \eqref{30} holds true for all $t \ge t_N$.
\end{proof}

Theorem \ref{T3} can easily be translated for the solution $(u,\rho)$ of \eqref{1}-\eqref{2}.

\begin{corollary}  \label{C}
Let $(u_0,\rho_0)$ be taken as \eqref{6'} and \eqref{9} and let $(u,\rho)$ be the global solution of \eqref{1}-\eqref{2}. Let $(\ol u,\ol\rho)$ be its $\omega$-limit obtained by Proposition \ref{T2}. Let $\Phi(u,\rho)$ be the Lyapunov function given by \eqref{16}. Then, as $t \to \infty$, $(u(t),\rho(t))$ converges to $(\ol u,\ol\rho)$ at the rate
\begin{multline*}
  \|u(t)-\ol u\|_{H^1_m{}'} + \|\rho(t)-\ol\rho\|_{L_2}
  \le (D\ve\theta)^{-1} [\Phi(u(t),\rho(t))-\Phi(\ol u,\ol\rho)]^\theta   \\
  \text{for all sufficiently large $t$}.
\end{multline*}
\end{corollary}

%%%%%%%%%%%%%%%%%%%%%%%%%%%%%%%%%%%%%%%%%%%%%%%%%%%%%%%%%%%%%
%%%%%%%%%%%%%%%%%%%%%%%%%%%%%%%%%%%%%%%%%%%%%%%%%%%%%%%%%%%%%
%%%%%%%%%%%%%%%%%%%%%%%%%%%%%%%%%%%%%%%%%%%%%%%%%%%%%%%%%%%%%
\section{Angle Condition}

We prove that the condition \eqref{19} is fulfilled.
\par

\noindent
I) Let $r' > 0$ be as in \eqref{33}. Since $\wt\Phi(V(t)) = \Phi(V(t))$ for $V(t) \in B^X(\ol V; r')$, it follows that $\frac d{dt}\wt\Phi(V(t)) = \frac d{dt}\Phi(V(t))$ for such $V(t)$'s. First, let us prove that
\begin{equation}  \label{31}
  \left|\tfrac d{dt}\wt\Phi(V(t))\right|
  \ge \ve'\|\wt\Phi'(V(t))\|_Z^2
  \qquad\text{if}\enskip V(t) \in B^X(\ol V; r'),
\end{equation}
with some constant $\ve' > 0$.
\par

In view of \eqref{5} and \eqref{21}, we observe that
\begin{equation*}
  \|\wt\Phi'(V(t))\|_Z^2
  \le C\{\|P_m[a\log(v(t)+f)-k\rho(t)]\|_{H^1_m}^2
       + \|\rho-cA_2^{-1}(v(t)+f)\|_{H^2_N}^2\}.
\end{equation*}
Furthermore, on account of \eqref{7} and \eqref{5}, we observe that
\begin{equation*}
  \|\wt\Phi'(V(t))\|_Z^2
  \le C\{\|[a\log(v(t)+f)-k\rho(t)]_x\|_{L_2}^2
       + \|A_2\rho-c(v(t)+f)\|_{L_2}^2\}.
\end{equation*}
Therefore, since $v(t)+f \ge \frac\delta2$ due to \eqref{34} and \eqref{33}, it follows that
\begin{equation*}
  \|\wt\Phi'(V(t))\|_Z^2
  \le C\{\|\sqrt{v(t)+f}\,[a\log(v(t)+f)-k\rho(t)]_x\|_{L_2}^2
         + \|\tfrac{d\rho}{dt}(t)\|_{L_2}^2\}.
\end{equation*}
Hence, in view of \eqref{13}, we conclude that \eqref{31} is satisfied.
\smallskip

\noindent
II) Next, let us prove that
\begin{equation}  \label{32}
  \left|\tfrac d{dt}\wt\Phi(V(t))\right|
  \ge \ve''\|\tfrac{dV}{dt}(t)\|_{Z^*}^2
  \qquad\text{if}\enskip V(t) \in B^X(\ol V; r'),
\end{equation}
with some constant $\ve'' > 0$.
\par

From \eqref{8}, we have $\|\tfrac{dV}{dt}(t)\|_{Z^*}^2 \le C\{\|\tfrac{dv}{dt}(t)\|_{{H^1_m}'}^2 + \|\tfrac{d\rho}{dt}(t)\|_{L_2}^2\}$. Here, noticing that $\tfrac{dv}{dt}(t) = \{(v(t)+f)[a\log(v(t)+f)-k\rho(t)]'\}'$, we apply Lemma \ref{L1} to observe that
\begin{equation*}
  \|\tfrac{dv}{dt}(t)\|_{{H^1_m}'}^2
     \le C\|(v(t)+f)[a\log(v(t)+f)-k\rho(t)]'\|_{L_2}^2.
\end{equation*}
Furthermore, on account of \eqref{14}, we observe that
\begin{equation*}
 \|\tfrac{dv}{dt}(t)\|_{{H^1_m}'}^2 \le C\|\sqrt{v(t)+f}\,
       [a\log(v(t)+f)-k\rho(t)]'\|_{L_2}^2.
\end{equation*}
Hence, in view of \eqref{13}, we conclude that \eqref{32} is satisfied.
\par

On account of \eqref{37}, it is clear that \eqref{31} and \eqref{32} yield the desired inequality \eqref{19}.

%%%%%%%%%%%%%%%%%%%%%%%%%%%%%%%%%%%%%%%%%%%%%%%%%%%%%%%%%%%%%%
%%%%%%%%%%%%%%%%%%%%%%%%%%%%%%%%%%%%%%%%%%%%%%%%%%%%%%%%%%%%%%
%%%%%%%%%%%%%%%%%%%%%%%%%%%%%%%%%%%%%%%%%%%%%%%%%%%%%%%%%%%%%%
\section{Gradient Inequality}

We next prove that the inequality \eqref{28} is fulfilled. It is necessary, however, to investigate much deeper properties of $\wt\Phi(V)$. Especially, we have to make use of analyticity of $\wt\Phi(V)$ in a suitable neighborhood of $\ol V$ in $X$.

%%%%%%%%%%%%%%%%%%%%%%%%%%%%%%%%%%%%%%%%%%%%%%%%%%%%%%%%%%%%%%
\subsection{Differentiability of $\wt\Phi'(U)$}

We begin with investigating differentiability of $\wt\Phi'(V)$.

\begin{proposition}  \label{P6}
The mapping $\wt\Phi'\,{:}\, X \to X$ is G\^ateaux differentiable at $\ol V$ with the derivative
\begin{equation}  \label{41}
  \wt\Phi''(\ol V)H = \begin{pmatrix}
    cP_m\left[\tfrac{ah}{\ol v+f} - k\eta\right]  \\
    k[\eta - cA_2^{-1}h]  \end{pmatrix}
  \qquad\text{for direction}\enskip 
      H = \begin{pmatrix} h \\ \eta \end{pmatrix}  \in  X,
\end{equation}
$\wt\Phi''(\ol V)$ being a bounded linear operator of $X$.
\end{proposition}

\begin{proof}
In view of \eqref{21}, we see that it suffices to verify that the mapping $v \mapsto \ell'(v+f)$ from $L_{2,m}(I)$ into $L_2(I)$ is G\^ateaux differentiable at $\ol v$.
\par

To see this, fix $h \in L_{2,m}(I)$ and let $0 < \theta \le 1$ be a variable. Then,
\begin{equation*}
  \ell'(\ol v+f+\theta h)-\ell'(\ol v+f)-\ell''(\ol v+f)\theta h
  = \theta\int_0^1
      [\ell''(\ol v+f+\vartheta\theta h)-\ell''(\ol v+f)]h\,d\vartheta.
\end{equation*}
Therefore,
\begin{multline*}
  \|\theta^{-1}[\ell'(\ol v+f+\theta h)-\ell'(\ol v+f)]
     - \ell''(\ol u+f) h\|_{L_2}  \\
  \le \left\|\int_0^1 
      [\ell''(\ol v+f+\theta\vartheta h)-\ell''(\ol v+f)]d\vartheta \,
       h\right\|_{L_2}.
\end{multline*}
Here, since $\sup_{-\infty < \xi < \infty} |\ell''(\xi)| < \infty$ due to \eqref{22}, the dominate convergence theorem is applicable to conclude that, as $\theta \to 0$,
\begin{equation*}
  \left\|\int_0^1
    [\ell''(\ol v+f+\theta\vartheta h)-\ell''(\ol v+f)]d\vartheta \,
     h\right\|_{L_2} \to 0.
\end{equation*}
Therefore, $v \mapsto \ell'(v+f)$ is G\^ateaux differentiable at $\ol v$ and its derivative for the direction $h$ is given by $\ell''(\ol v+f)h$. It is also noticed by \eqref{22} that $\ell''(\ol v+f) = \frac1{\ol v+f}$.
\par

The operator expressed by \eqref{41} is obviously a bounded linear operator of $X$.
\end{proof}

It is impossible to show that the mapping $v \mapsto \ell'(v+f)$ is Fr\'echet differentiable from $L_{2,m}(I)$ into $L_2(I)$; but, it is certainly possible to see that the mapping is so from $\cl C_m(\ol I)$ into $\cl C(\ol I)$. This fact then leads us to introduce another underlying space
\begin{equation}  \label{20}
  Y = \left\{\begin{pmatrix}  v  \\  \rho  \end{pmatrix};\;
      v \in \cl C_m(\ol I) \quad\text{and}\quad
      \rho \in H^1(I)\right\},
\end{equation}
here $\cl C_m(\ol I)$ is a closed subspace of $\cl C(\ol I)$ consisting of zero-mean functions of $\cl C(\ol I)$. By \eqref{5} and \eqref{17}, we see that this Banach space $Y$ is such that $Z \subset Y \subset X$ and the intermediate inequality
\begin{equation}  \label{38}
  \|V\|_Y \le C\|V\|_Z^\kappa \|V\|_X^{1-\kappa}, \qquad V \in Z,
\end{equation}
is valid with some $0 < \kappa < 1$. In addition, it is obvious that
\begin{equation}  \label{44}
  \ol V \in \omega(V) \subset \ol B^Z(0; R) \subset Y.
\end{equation}

\begin{proposition}  \label{P7}
If $V \in Y$, then $\wt\Phi'(V) \in Y$. And the mapping $\wt\Phi'\,{:}\, Y \to Y$ is continuously Fr\'echet differentiable with the derivative
\begin{equation}  \label{49}
  \wt\Phi''(V)H = \begin{pmatrix}
    cP_m[a\ell''(v+f)h - k\eta]  \\
    k[\eta - cA_2^{-1}h]  \end{pmatrix},
  \qquad  V = \begin{pmatrix} v \\ \rho \end{pmatrix},\,
         H = \begin{pmatrix} h \\ \eta \end{pmatrix}  \in Y.
\end{equation}
\end{proposition}

\begin{proof}
Since the projection $P_m$ given by \eqref{25} is a bounded linear operator from $\cl C(\ol I)$ into $\cl C_m(\ol I)$, it is easy to see that $V \in Y$ implies $\wt\Phi'(V) \in Y$. In addition, we see that it suffices to verify that the mapping $v \mapsto \ell'(v+f)$ from $\cl C_m(\ol I)$ into $\cl C(\ol I)$ is continuously Fr\'echet differentiable.
\par

To see this, let $v,\, h \in \cl C_m(\ol I)$. Then, the same calculations as in the proof of Proposition \ref{P6} yield the equality
\begin{equation*}
  \ell'(v+f+h)-\ell'(v+f)-\ell''(v+f)h
  = \int_0^1 [\ell''(v+f+\vartheta h)-\ell''(v+f)]h\,d\vartheta.
\end{equation*}
Because of $\sup_{-\infty<\xi<\infty} |\ell'''(\xi)| < \infty$, it follows that
\begin{equation*}
  \|\ell'(v+f+h)-\ell'(v+f)-\ell''(v+f)h\|_{\cl C}
  \le C\|h\|_{\cl C_m}^2,
\end{equation*}
which shows that $v \mapsto \ell'(v+f)$ is Fr\'echet differentiable and its derivative at $v$ is given by $h \mapsto \ell''(v+f)h$. Hence, \eqref{49} is obtained.
\par

Continuity of the mapping $\wt\Phi''\,{:}\, Y \to \cl L(Y)$ is also easily verified.
\end{proof}

%%%%%%%%%%%%%%%%%%%%%%%%%%%%%%%%%%%%%%%%%%%%%%%%%%%%
\subsection{Some Properties of $\wt\Phi''(\ol V)$}

Put $L=\wt\Phi''(\ol V)$. By Propositions \ref{P6} and \ref{P7}, $L$ is a bounded linear operator of both $X$ and $Y$, i.e., $L \in \cl L(X) \cap \cl L(Y)$. In this subsection, we want to investigate other various properties of $L$.
\smallskip

\noindent
(I) $L$ is a symmetric operator of $X$. In fact, we have
\begin{multline*}
  (LV,W)_X = (cP_m[\tfrac{av}{\ol v+f}-k\rho],w)_{L_{2,m}}
             + (k[\rho-cA_2^{-1}v],\varrho)_{H^1},  \\
           V=\begin{pmatrix}v \\ \rho\end{pmatrix},\;
           W=\begin{pmatrix}w \\ \varrho \end{pmatrix} \in X.
\end{multline*}
Since $(P_m[\tfrac{av}{\ol v+f}],w)_{L_{2,m}} = (\tfrac{av}{\ol v+f},w)_{L_2} = (v,P_m[\tfrac{aw}{\ol v+f}])_{L_{2,m}}$ and since $(P_m\rho,w)_{L_{2,m}} = (\rho,w)_{L_2} \linebreak = (A_2^\frac12\rho,A_2^\frac12A_2^{-1}w)_{L_2} = (\rho,A_2^{-1}w)_{H^1}$ due to \eqref{4} and \eqref{23}, we verify that $(LV,W)_X = (V,LW)_X$ for any $V,\, W \in X$.
\smallskip

\noindent
(II) $L$ is a Fredholm operator of $X$. In fact, we have $LV = \begin{pmatrix} acP_m(\frac v{\ol v+f}) \\ k\rho \end{pmatrix} - ck \begin{pmatrix} P_m\rho \\ A_2^{-1}v \end{pmatrix}$ for $V = \begin{pmatrix} v \\ \rho \end{pmatrix} \in X$. Therefore, the result is obtained by the Riesz-Schauder theory (see, e.g. \cite[Section 6.2]{Bre} or \cite[Section 9-3]{Tan}) and the following lemma.

\begin{lemma} \label{L2}
The mapping $v \mapsto P_m(\frac v{\ol v+f})$ is a linear isomorphism from $L_{2,m}(I)$ onto itself.
\end{lemma}

\begin{proof}
Put $\ol a(x)=\frac1{\ol v(x)+f}$. Then, since $(\|\ol v\|_{\cl C}+f)^{-1} \le \ol a(x) \le \delta^{-1}$ on $\ol I$ due to \eqref{34}, it is seen that $P_m[\ol a(x)v]=w$ if and only if $v = \ol a(x)^{-1}w -\left[\tfrac{\int_I \ol a(\xi)^{-1}w(\xi)d\xi} {\int_I \ol a(\xi)^{-1}d\xi}\right]\ol a(x)^{-1}$ for $v,\, w \in L_{2,m}(I)$. This means that $v \mapsto P_m[\ol a(x)v]$ admits a bounded inverse on $L_{2,m}(I)$.
\end{proof}

\noindent
(III) $L$ enjoys the following range condition with respect to $Y$:
\begin{equation}  \label{50}
  \text{if $LV \in Y$ for $V \in X$, then actually $V \in Y$,
        i.e., $L^{-1}(Y) \subset Y$.}
\end{equation}
By the similar argument as in the proof of Lemma \ref{L2}, we see that $P_m[\frac v{\ol v+f}] \in \cl C_m(\ol I)$ implies $v \in \cl C_m(\ol I)$. From this fact, the proof of \eqref{50} is quite easy.
\smallskip

\noindent
(IV) $L$ induces an orthogonal decomposition of $X$ into the form $X = \cl K(L) + L(X)$. In fact, as $L$ is a Fredholm operator of $X$, the kernel $\cl K(L)$ is of finite dimension, say ${\rm dim}\,\cl K(L) = N$, and the range $L(X)$ is a closed subspace of $X$ having a finite codimension. Let $P\,{:}\, X \to \cl K(L)$ be the orthogonal projection from $X$ onto $\cl K(L)$. Then, $X$ is orthogonally decomposed into $X = PX + (1-P)X$, $PX$ coinciding with $\cl K(L)$. We here observe that $(1-P)X = L(X)$. Indeed, let $V \in L(X)$, i.e., $V=LW$; then, since
\begin{equation} \label{53}
   LP = 0  \enskip\text{and}\enskip  PL = P^*L^* = (LP)^* = 0,
\end{equation}
it follows that $PV = PLW=0$; hence, $V = (1-P)V \in (1-P)X$, that is, $L(X) \subset (1-P)X$. Assume that $L(X) \varsubsetneqq (1-P)X$; then, there exists $0 \not= V \in (1-P)X$ which is orthogonal to $\ol{L(X)}=L(X)$; since $0=(V,LW)_X = (LV,W)_X$ for any $W \in X$, $LV$ must be $0$, i.e., $V \in PX$; hence, $V \in PX \cap (1-P)X = \{0\}$ which contradicts the assumption $V \not=0$.
\par

In particular, $L$ is an isomorphism from $L(X)$ onto itself.
\smallskip

\noindent
(V) $L$ equally induces a topological decomposition of $Y$ into the form $Y = \cl K(L) + L(Y)$. First, let us notice from \eqref{50} that $LV=0$ for $V \in X$ implies $V \in Y$, i.e., $\cl K(L) \subset Y$. Next, as $\cl K(L)$ is of finite dimension, the induced topology of $\cl K(L)$ from $Y$ is equivalent to that from $X$. In other words, the operator $P\,{:}\, Y \to \cl K(L)$ is continuous with respect to the $Y$-norm. As $P$ keeps the property $P^2 = P$ on $Y$, $P$ can then induce a topological decomposition of $Y$, too, into $Y = PY + (1-P)Y$, $PY$ coinciding with $\cl K(L)$. We here observe that $(1-P)Y = L(Y)$. Indeed, if $V \in (1-P)Y\, (\subset (1-P)X)$, then $V \in L(X)$ and $V = LW$ with some $W \in X$; but, \eqref{50} gives that $W \in Y$ and $V \in L(Y)$. Conversely, if $V \in L(Y)\, (\subset L(X))$, then $V = (1-P)W$ with some $W \in X$; but, $W = V+PW \in Y$ (due to $L \in \cl L(Y)$); hence, $V \in (1-P)Y$. In particular, $L(Y)$ is a closed subspace of $Y$.

\begin{proposition}  \label{P8}
The operator $L$ is an isomorphism from $L(Y)$ onto itself.
\end{proposition}

\begin{proof}
We already know that $L(Y)$ is a Banach subspace of $Y$ and that $L$ is continuous and bijective from $L(Y)$ onto itself. Then, by virtue of a corollary of the open mapping theorem (see \cite[p.\,77]{Yos}), we know that its inverse is also continuous.
\end{proof}

%%%%%%%%%%%%%%%%%%%%%%%%%%%%%%%%%%%%%%%%%%%%%%%%%%%%%%%%%%%%%
\subsection{Critical Manifold}

Following the idea due to Chill \cite{Ch}, we here introduce the critical manifold of $\wt\Phi(V)$ but is handled in the underlying space $Y$. Set a surface $S$ in $Y$ by
\begin{equation}  \label{42}
  S = \{V \in Y;\; \wt\Phi'(V) \in \cl K(L)\}
    = \{V \in Y;\; (1-P)\wt\Phi'(V)=0\}.
\end{equation}
By Critical Condition and \eqref{44} we know that $\wt\Phi'(\ol V) = 0$ with $\ol V \in Y$, thereby $\ol V$ certainly lies on $S$.
\par

Then, in a neighborhood of $\ol V$, $S$ is seen to be a $\cl C^1$-manifold of dimension $N (= {\rm dim}\, \cl K(L)$). Indeed, identifying $Y$ with the product space $\cl K(L)\times L(Y)$, set an operator $G\,{:}\, Y \to L(Y)$ such that $G(V^0,V^1) = (1-P)\wt\Phi'(V^0+V^1)$ for $(V^0,V^1) \in Y$ and consider the equation $G(V^0,V^1)=0$ for $(V^0,V^1)$. According to Propositions \ref{P7}, we have $G_{V^1}(V^0,V^1) = (1-P)\wt\Phi''(V^0+V^1)_{|L(Y)}$. Therefore, Proposition \ref{P8} shows that $G_{V^1}(P\ol V,(1-P)\ol V) = L_{|L(Y)}$ is an isomorphism of $L(Y)$. So, the implicit function theorem is applicable to $G(V^0,V^1)=0$ to conclude that in some neighborhood $\cl U(\ol V) = \cl U^0(P\ol V) \times \cl U^1((1-P)\ol V)$ of $\ol V$, where $\cl U^0(P\ol V)$ (resp. $\cl U^1((1-P)\ol V)$) is a neighborhood of $P\ol V$ (resp. $(1-P)\ol V$) in the space $\cl K(L)$ (resp. $L(Y)$), the unique solution to $G(V^0,V^1)=0$ is given by $(V^0,g(V^0))$, where $g\,{:}\, \cl U(P\ol V) \to \cl U((1-P)\ol V)$ is a $\cl C^1$ mapping such that $g(P\ol V) = (1-P)\ol V$. As a consequence, $S$ is seen to be represented as
\begin{equation*}
  S \cap \cl U(\ol V)
  = \{V^0+g(V^0);\; V^0 \in \cl U^0(P\ol V)
    \enskip\text{and}\enskip 
    g\,{:}\, \cl U^0(P\ol V) \to L(Y)\}.
\end{equation*}
Moreover, let $V^0_1,\ldots,V^0_N$ be an orthogonal basis of $\cl K(L)$ and introduce a coordinate to $\cl K(L)$ by the correspondence
\begin{equation*}
  V^0 = \xi_1V^0_1+\cdots+\xi_NV^0_N \in \cl K(L) \quad  \longleftrightarrow 
      \quad  \bs\xi = (\xi_1,\ldots,\xi_N) \in \bb R^N.
\end{equation*}
Especially, denote $P\ol V = \sum_{n=1}^N \ol\xi_nV^0_n \leftrightarrow \ol{\bs\xi}=(\ol\xi_1,\ldots,\ol\xi_N)$. Let $\Omega \subset \bb R^N$ be the region corresponding to $\cl U^0(P\ol V)$. Then, we can define to $S \cap \cl U(\ol V)$ a coordinate by the relation
\begin{equation*}
  V^S = \textstyle\sum_{n=1}^N \xi_nV^0_n
     + g\left(\textstyle\sum_{n=1}^N \xi_nV^0_n\right) \in S \cap \cl U(\ol V)
  \quad  \longleftrightarrow  \quad  \bs\xi = (\xi_1,\ldots,\xi_N) \in \Omega.
\end{equation*}
\par

We often identify the vector $V^0 \in \cl U^0(P\ol V)$ with its coordinate $\bs\xi \in \Omega$. Then, $g$ is considered to be a $\cl C^1$ function for $\bs\xi \in \Omega$ with values in $L(Y)$.
\par

Furthermore, we can verify that $\wt\Phi(V)$ is analytic in a neighborhood of $\ol V$ on $S$.

\begin{proposition}  \label{P10}
The function $\bs\xi \mapsto \wt\Phi\big(\bs\xi + g(\bs\xi)\big)$ is analytic in $\Omega$ if $\Omega$ is replaced by a sufficiently smaller one.
\end{proposition}

\begin{proof}
We will employ the methods of complexification by extending every thing concerning the manifold $S$ to the complex one in a suitable complex neighborhood of $\ol V$.
\par

The underlying space is extended to $X_\bb C = X + iX$. It is similar for $Y$, i.e., $Y_\bb C = Y+iY$. The derivation $\wt\Phi'(V)$ given by \eqref{21} can naturally be extended for complex $V$ as follows. The operator $A_2$ can be a realization of $-b\frac{d^2}{dx^2}+d$ under the homogeneous Neumann boundary conditions in the complex space $L_2(I; \bb C)$ and $A_2$ becomes a positive definite self-adjoint operator of $L_2(I; \bb C)$. The logarithmic function can be extended as an analytic function in the complex domain  $\bb C-(-\infty,\frac\delta2]$. So, on account of \eqref{34}, if $r_{\bb C} > 0$ is sufficiently small, then $\log(v+f)$ is defined for any $v \in \cl C_m(\ol I; \bb C)$ such that $\|v-\ol v\|_{\cl C} < r_{\bb C}$. Therefore, in view of \eqref{22}, the operator $\wt\Phi'(V)$ can be extended as
\begin{equation*}
  [\wt\Phi']_{\bb C}(V) = \begin{pmatrix} cP_m[a\log(v+f)-k\rho]  \\
                k[\rho-cA_2^{-1}(v+f)] \end{pmatrix},
  \qquad V = \begin{pmatrix} v \\ \rho \end{pmatrix}
           \in B^{Y_\bb C}(\ol V; r_{\bb C}).
\end{equation*}
\par

By the same proof as for Proposition \ref{P7}, the mapping $[\wt\Phi']_{\bb C}\,{:}\, B^{Y_{\bb C}}(\ol V;r_{\bb C}) \to Y_{\bb C}$ is seen to be Fr\'echet differentiable with the derivative $[\wt\Phi']_{\bb C}'(V)$ which takes the same form as $\wt\Phi''(V)$ expressed by \eqref{49}. Put $L_{\bb C}= [\wt\Phi']_{\bb C}'(\ol V)$. Since $L_{\bb C}$ takes the same form as $\wt\Phi''(\ol V)$ expressed by \eqref{41}, $L_{\bb C}$ is a bounded linear operator of $X_{\bb C}$ which can be given by $L_{\bb C}V = L({\rm Re}\,V)+iL({\rm Im}\,V)$ for $V \in X_{\bb C}$, which means that $L_{\bb C}$ is a real operator of $X_{\bb C}$ and can be identified with $L$. Therefore, $L$ is a symmetric Fredholm operator of $X_{\bb C}$, too. In addition, since $V \in \cl K(L)$ if and only if ${\rm Re}\,V,\; {\rm Im}\,V \in \cl K(L)$, the basis $V_1^0,\ldots,V_N^0$ fixed above in the real $\cl K(L)$ can span the complex $\cl K(L)$ in $X_{\bb C}$.
\par

Meanwhile, let $P_{\bb C}$ be the orthogonal projection from $X_{\bb C}$ onto $\cl K(L)$. Similarly, we have $P_{\bb C}V=P({\rm Re}\,V)+iP({\rm Im}\,V)$ for $V \in X_{\bb C}$, which means that $P_{\bb C}$ is also a real operator of $X_{\bb C}$ and can be identified with $P$. Therefore, $X_{\bb C}$ is orthogonally decomposed into $X_{\bb C}=PX_{\bb C}+(1-P)X_{\bb C}$, where $PX_{\bb C}=\cl K(L)$ and $(1-P)X_{\bb C}=L(X_{\bb C})$. As the similar range condition to \eqref{50} is satisfied by $L$ in $Y_{\bb C}$, this decomposition induces the topological decomposition of $Y_{\bb C}$ into $Y_{\bb C}=PY_{\bb C}+(1-P)Y_{\bb C}$, where $PY_{\bb C}=\cl K(L)$ and $(1-P)Y_{\bb C}=L(Y_{\bb C})$. By the same proof as for Proposition \ref{P8}, we know that $L$ is an isomorphism from $L(Y_{\bb C})$ onto itself.
\par

The critical manifold in $Y_\bb C$ is now defined by
\begin{equation*}
  S_\bb C = \{V \in B^{Y_\bb C}(\ol V; r_{\bb C});\;
            (1-P)[\wt\Phi']_{\bb C}(V)=0\}.
\end{equation*}
The implicit function theorem of complex version provides that, in some neighborhood $\cl U_{\bb C}(\ol V) = \cl U^0_\bb C(P\ol V) \times \cl U^1_\bb C((1-P)\ol V)$, where $\cl U^0_\bb C(P\ol V)$ (resp. $\cl U^1_\bb C((1-P)\ol V)$) is a neighborhood of $P\ol V$ (resp. $(1-P)\ol V$) in the space $\cl K(L)$ (resp. $L(Y_\bb C)$), the vector of $S_\bb C \cap \cl U_\bb C(\ol V)$ is represented by $V^0+g_\bb C(V^0)$, where $V^0 \in \cl U^0_\bb C(P\ol V)$ and $g_\bb C\,{:}\, \cl U^0_\bb C(P\ol V) \to L(Y_\bb C)$ is a $\cl C^1$ mapping. Moreover, using the complex coordinate $V^0 = \xi_1V^0_1+\cdots+\xi_NV^0_N \in \cl K(L) \leftrightarrow  \bs\xi = (\xi_1,\ldots,\xi_N) \in \bb C^N$, the vectors $V^S \in S_\bb C \cap \cl U_\bb C(\ol V)$ are expressed as $V^S = \sum_{n=1}^N \xi_nV_n^0 + g_\bb C(\sum_{n=1}^N \xi_nV_n^0) \leftrightarrow \bs\xi = (\xi_1,\ldots,\xi_n) \in \Omega_\bb C$, $\Omega_\bb C$ being a region of $\bb C^N$ which corresponds to $\cl U_\bb C(\ol V)$. Therefore, $g_\bb C$ is a $\cl C^1$ function defined in $\Omega_\bb C$ with values in $L(Y_\bb C)$.
\par

Finally, we complexificate the Lyapunov function $\wt\Phi(V)$ as
\begin{multline*}
  \wt\Phi_\bb C(V)
  = \int_I \Bigg\{ac[(v+f)\log(v+f)-(v+f)]
      + \frac{bk}2(\rho')^2  \\
      + \frac{dk}2\rho^2 - ck(v+f)\rho\Bigg\}dx,
  \qquad V \in B^{Y_\bb C}(\ol V; r_\bb C).
\end{multline*}
By the same arguments as in the proof of Proposition \ref{P3}, $\wt\Phi_\bb C{:}\, B^{Y_\bb C}(\ol V; r_\bb C) \to \bb C$ is continuously Fr\'echet differentiable. Therefore, the composed function $\bs\xi \mapsto V^S \mapsto \wt\Phi_\bb C(V^S)$ is continuously differentiable in $\Omega_\bb C$. Let us here recall the characterization of analytic functions of several complex variables. According to \cite[(9.10.1)]{Di} or \cite[Theorem 2.2.8]{Hor}, a continuously differentiable function from a region of $\bb C^N$ is an analytic function.
\par

In this way, we have established the desired analyticity.
\end{proof}

%%%%%%%%%%%%%%%%%%%%%%%%%%%%%%%%%%%%%%%%%%%%%%%%%%%%%%%%%%%%%%%%%%
\subsection{Verification of \eqref{28}}

It is now ready to prove that \eqref{28} is fulfilled.
\par

We use a decomposition of the vectors $V \in \cl U(\ol V)$ given by the form $V=V^S+V^1$, where $V^S \in S$ and $V^1 \in L(Y)$. In fact, express $V$ as
\begin{equation*}
  V = [PV+g(PV)] + [(1-P)V-g(PV)] \equiv V^S + V^1,
  \qquad  V \in \cl U(\ol V),
\end{equation*}
where $PV+g(PV)=(PV,g(PV)) \in S$ and $(1-P)V-g(PV) \in L(Y)$. It is clear that
\begin{equation}  \label{46}
  \text{as}\enskip V \to \ol V \enskip\text{in}\enskip Y,\quad V^S \to \ol V
  \enskip\text{and}\enskip  V^1 \to 0 \enskip\text{in}\enskip Y.
\end{equation}
\par

This decomposition yields the following important estimate for $\wt\Phi'(V)$.

\begin{proposition}
It holds that
\begin{gather}  \label{47}
  \|P\wt\Phi'(V)\|_Y \ge \|\wt\Phi'(V^S)\|_Y - o(1)\|V^1\|_Y,
      \qquad  V = V^S + V^1 \in \cl U(\ol V),  \\  \label{52}
  \|(1-P)\wt\Phi'(V)\|_Y \ge c\|V^1\|_Y,
      \qquad  V = V^S + V^1 \in \cl U(\ol V),
\end{gather}
if $\cl U(\ol V)$ is replaced by a sufficiently smaller one. Here, $o(1)$ stands for a small quantity tending to $0$ as $V \to \ol V$ in $Y$ while $c$ is a fixed positive constant.
\end{proposition}

\begin{proof}
We have
\begin{align*}
  \wt\Phi'(V) &= \wt\Phi'(V^S) + \wt\Phi''(V^S)(V-V^S) + r(V-V^S)  \\
    &= \wt\Phi'(V^S) + [\wt\Phi''(V^S)-\wt\Phi''(\ol V)]V^1
       + \wt\Phi''(\ol V)V^1 + r(V^1),
\end{align*}
where $\|r(V^1)\|_Y = o(\|V^1\|_Y)$. As $P\wt\Phi'(V^S)= \wt\Phi'(V^S)$ and $P\wt\Phi''(\ol V)=0$ due to \eqref{53}, we observe that
\begin{equation*}
  P\wt\Phi'(V) = \wt\Phi'(V^S)
     + P[\wt\Phi''(V^S)-\wt\Phi''(\ol V)]V^1 + Pr(V^1).
\end{equation*}
Therefore,
\begin{equation*}
  \|P\wt\Phi'(V)\|_Y \ge \|\wt\Phi'(V^S)\|_Y - o(1)\|V^1\|_Y - o(\|V^1\|_Y),
\end{equation*}
where $o(1) = \|\wt\Phi''(V^S)-\wt\Phi''(\ol V)\|_{\cl L(Y)}$ tends to $0$ as $V \to \ol V$ in $Y$. Hence, \eqref{47} is observed.
\par

On the other hand, by the same reasons, we have
\begin{equation*}
  (1-P)\wt\Phi'(V) = (1-P)[\wt\Phi''(V^S)-\wt\Phi''(\ol V)]V^1
     + (1-P)\wt\Phi''(\ol V)V^1 + (1-P)r(V_1).
\end{equation*}
Therefore,
\begin{equation*}
  \|(1-P)\wt\Phi'(V)\|_{L(Y)} \ge \|(1-P)\wt\Phi''(\ol V)V^1\|_{L(Y)}
  - o(1)\|V^1\|_{L(Y)}-o(\|V^1\|_Y).
\end{equation*}
Here, since Proposition \ref{P8} yields the estimate
\begin{align*}
  \|V^1\|_{L(Y)} &= \|L^{-1}LV^1\|_{L(Y)}
     \le \|L^{-1}\|_{\cl L(L(Y))}\|LV^1\|_{L(Y)}   \\
     &\le \|L^{-1}\|_{\cl L(L(Y))}\|(1-P)\wt\Phi''(\ol V)V^1\|_{L(Y)},
\end{align*}
it follows that
\begin{equation*}
  \|(1-P)\wt\Phi'(V)\|_Y
  \ge \|L^{-1}\|^{-1}\|V^1\|_Y - o(1)\|V^1\|_Y
      - o(\|V^1\|_Y).
\end{equation*}
Hence, \eqref{52} is observed.
\end{proof}

Since $\|\wt\Phi'(V)\|_Y \ge C[\|P\wt\Phi'(V)\|_Y + \|(1-P)\wt\Phi'(V)\|_Y]$, it immediately follows from \eqref{46}, \eqref{47} and \eqref{52} that
\begin{equation}  \label{45}
  \|\wt\Phi'(V)\|_Y \ge \tilde c[\|\wt\Phi'(V^S)\|_Y + \|V^1\|_Y]
  \qquad\text{for all}\enskip V = V^S+V^1 \in \cl U(\ol V),
\end{equation}
with some positive constant $\tilde c$ if $\cl U(\ol V)$ is replaced by a sufficiently smaller one.
\smallskip

\noindent
I) {\sl Gradient Inequality in $L(Y)$.}
Because of
\begin{equation}  \label{43}
  |\wt\Phi(V)-\wt\Phi(\ol V)|
  \le |\wt\Phi(V)-\wt\Phi(V^S)| + |\wt\Phi(V^S)-\wt\Phi(\ol V)|,
\end{equation}
what we have to do is to estimate the two differences in the right hand side. First, let us estimate the former.
\par

On account of Proposition \ref{P7}, we can write as
\begin{align*}
  \wt\Phi(V) - \wt\Phi(V^S)
  &= (\wt\Phi'(V^S),V-V^S) + \tfrac12\big(\wt\Phi''(V^S)(V-V^S),V-V^S\big)
     + o(\|V-V^S\|_Y^2)  \\
  &= (\wt\Phi'(V^S),V^1) + \tfrac12(\wt\Phi''(V^S)V^1,V^1) + o(\|V^1\|_Y^2).
\end{align*}
Therefore,
\begin{align*}
  |\wt\Phi(V)-\wt\Phi(V^S)|
  &\le \|\wt\Phi'(V^S)\|_X\|V^1\|_X + C\|V^1\|_Y\|V^1\|_X + o(\|V^1\|_Y^2) \\
  &\le \|\wt\Phi'(V^S)\|_X^2 + C\|V^1\|_Y^2.
\end{align*}
Hence, \eqref{45} provides the inequality
\begin{equation}  \label{48}
  \|\wt\Phi'(V)\|_Y \ge C|\wt\Phi(V)-\wt\Phi(V^S)|^\frac12,
  \qquad    V=V^S+V^1 \in \cl U(\ol V).
\end{equation}
\smallskip

\noindent
II) {\sl Gradient Inequality on $S$.} Due to Proposition \ref{P10}, the function $\phi(\bs\xi) \equiv \wt\Phi\big(\bs\xi + g(\bs\xi)\big)$ is analytic in a region $\Omega \subset \bb R^N$ containing $\ol{\bs\xi}$. Then, we can apply the classical result due to {\L}ojasiewicz \cite{Lo} to conclude that there exists some exponent $0 < \theta \le \frac12$ for which it holds that
\begin{equation*}
  \|\nabla_{\bs\xi}\phi(\bs\xi)\|_{\bb R^N}
  \ge C|\phi(\bs\xi) - \phi(\ol{\bs\xi})|^{1-\theta},
  \qquad  \bs\xi \in \Omega,
\end{equation*}
if $\Omega$ is replaced by a sufficiently smaller one. Here, it is seen by Proposition \ref{P3} that 
\begin{equation*}
  \nabla_{\bs\xi}\phi(\bs\xi)
  = [(\wt\Phi'\big(\bs\xi+g(\bs\xi)\big),V^0_n+g'(\bs\xi)V^0_n)_X]_{n=1,\ldots,N}.
\end{equation*}
Therefore,
\begin{equation*}
  \|\nabla_{\bs\xi}\phi(\bs\xi)\|_{\bb R^N}
  \le C\|\wt\Phi'(\bs\xi+g(\bs\xi))\|_X = C\|\wt\Phi'(V^S)\|_X.
\end{equation*}
By virtue of \eqref{45}, we obtain that
\begin{equation}  \label{51}
  \|\wt\Phi'(V)\|_Y \ge C\|\wt\Phi'(V^S)\|_Y
  \ge C|\wt\Phi(V^S) - \wt\Phi(\ol V)|^{1-\theta},
  \qquad    V=V^S+V^1 \in \cl U(\ol V),
\end{equation}
if $\cl U(\ol V)$ is replaced by a sufficiently smaller one.
\smallskip

\noindent
III) {\sl Completion of the Proof.} We obtain from \eqref{43}, \eqref{48} and \eqref{51} that
\begin{equation*}
  \|\wt\Phi'(V)\|_Y \ge C|\wt\Phi(V) - \wt\Phi(\ol V)|^{1-\theta},
  \qquad    V \in \cl U(\ol V).
\end{equation*}
\par

Let us ultimately verify \eqref{28}. In view of \eqref{38}, there is a radius $r'''>0$ such that $B^X(\ol V; r''') \cap B^Z(0; R) \subset \cl U(\ol V)$, $R$ being the constant in \eqref{14}. It then follows that
\begin{equation*}
  \|\wt\Phi'(V(t))\|_Y \ge C|\wt\Phi(V(t)) - \wt\Phi(\ol V)|^{1-\theta}
  \quad\text{if}\enskip   V(t) \in B^X(\ol V; r''').
\end{equation*}
Hence, as $\|\wt\Phi'(V(t))\|_Z \ge C\|\wt\Phi'(V(t))\|_Y$ due to \eqref{5} and \eqref{20}, we conclude \eqref{28}.

\bigskip\bigskip

Satoru Iwasaki \par
Department of Information and Physical Sciences  \par
Graduate School of Information Science and Technology  \par
Osaka University  \par
Yamadaoka, Suita, 565-0871  \par
Japan  \par
E-mail: satoru.iwasaki@ist.osaka-u.ac.jp
\bigskip\bigskip

Koichi Osaki  \par
Department of Mathematical Sciences  \par
School of Science and Technology  \par
Kwansei Gakuin University  \par
Gakuen, Sanda, 669-1337  \par
Japan  \par
E-mail: osaki@kwansei.ac.jp
\bigskip\bigskip

Atsushi Yagi  \par
Professor Emeritus  \par
Graduate School of Information Science and Technology  \par
Osaka University  \par
Yamadaoka, Suita, 565-0871  \par
Japan  \par
E-mail: yagi-atsushi-ch@alumni.osaka-u.ac.jp

\end{document}